\documentclass[reqno,11pt]{article}
\usepackage{a4wide}
\usepackage{color,eucal,enumerate,mathrsfs, amsthm}
\usepackage[normalem]{ulem}
\usepackage{amsmath,amssymb,epsfig,bbm}
\numberwithin{equation}{section}
\usepackage{tikz}

\usepackage[pdfborder={0 0 0}]{hyperref}  

\usepackage[latin1]{inputenc}
\usepackage[english]{babel}

\newcommand{\dist}{\mathsf{d}}
\newcommand{\Lint}{\mathit{L}}
\newcommand{\Sp}{\mathrm{X}}
\newcommand{\mea}{\mathfrak{m}}

\newcommand{\mms}{(\Sp, \dist, \mea)}

\newcommand{\real}{\mathbb{R}}

\newcommand{\W}{\mathit{W}^{1,2}}

\newcommand{\D}{\mathrm{D}}

\newcommand{\loc}{\mathsf{loc}}

\newcommand{\C}{\mathit{C}}

\newcommand{\mae}{\mea\text{-a.e.}}

\def\Xint#1{\mathchoice 
  {\XXint\displaystyle\textstyle{#1}}%
  {\XXint\textstyle\scriptstyle{#1}}%
  {\XXint\scriptstyle\scriptscriptstyle{#1}}%
  {\XXint\scriptscriptstyle\scriptscriptstyle{#1}}%
  \!\int} 
\def\XXint#1#2#3{{\setbox0=\hbox{$#1{#2#3}{\int}$} 
  \vcenter{\hbox{$#2#3$}}\kern-.5\wd0}} 
\def\-int{\Xint -}

\numberwithin{equation}{subsection}

\newcommand{\R}{\mathbb{R}}

\newcommand{\mm}{{\mbox{\boldmath$m$}}}

\newcommand{\sfd}{{\sf d}}

\newcommand{\Kliminf}{K\kern-3pt-\kern-2pt\mathop{\rm lim\,inf}\limits}  
\newcommand{\supp}{\mathop{\rm supp}\nolimits}   
\newcommand{\Lip}{\mathop{\rm Lip}\nolimits}          
\renewcommand{\d}{{\mathrm d}}

\newcommand{\restr}[1]{\lower3pt\hbox{$|_{#1}$}} 

\newcommand{\la}{\left<}                  
\newcommand{\ra}{\right>}
\newcommand{\eps}{\varepsilon}  
\newcommand{\nchi}{{\raise.3ex\hbox{$\chi$}}}

\newcommand{\fr}{\penalty-20\null\hfill$\blacksquare$}                      

\renewcommand{\mm}{\mathfrak m}                                

\renewenvironment{proof}{\removelastskip\par\medskip   
\noindent{\em proof} \rm}{\penalty-20\null\hfill$\square$\par\medbreak}

\newtheorem{theorem}{Theorem}[section]

\newtheorem{lemma}[theorem]{Lemma}

\newtheorem{definition}[theorem]{Definition}

\newtheorem{remark}[theorem]{Remark}

\newcommand{\bd}{{\mathbf\Delta}}

\newcommand{\X}{{\rm X}}

\newcommand{\CD}{{\sf CD}}
\newcommand{\RCD}{{\sf RCD}}

\newcommand{\HS}{{\lower.3ex\hbox{\scriptsize{\sf HS}}}}

\makeindex
\title{A note about the strong maximum principle on  $\RCD$ spaces}
\begin{document}

\author{
   Nicola Gigli\
   \thanks{SISSA. email: \textsf{ngigli@sissa.it}} \and Chiara Rigoni\thanks{SISSA. crigoni@sissa.it}
   }
\maketitle
\begin{abstract} We give a quick and direct proof of the strong maximum principle on finite dimensional $\RCD$ spaces based on the Laplacian comparison of the squared distance.
\end{abstract}

\tableofcontents

\section{Introduction}
In the context of analysis in metric measure spaces it is by now well understood that a doubling condition and a Poincar\'e inequality are sufficient to derive the basics of elliptic regularity theory. In particular, one can obtain the Harnack inequality for harmonic functions which in turns implies the strong maximum principle. We refer to \cite{Bjorn-Bjorn11} for an overview on the topic and detailed bibliography.

$\RCD^*(K,N)$ spaces (\cite{AmbrosioGigliSavare11-2}, \cite{Gigli12}, see also \cite{AmbrosioGigliSavare12}, \cite{Erbar-Kuwada-Sturm13}, \cite{AmbrosioMondinoSavare13})  are, for finite $N$, doubling (\cite{Sturm06II}) and supporting a Poincar\'e inequality (\cite{Rajala12}) and thus in particular the above applies. Still, given that in fact such spaces are much more regular than general doubling\&Poincar\'e ones, one might wonder whether there is a simpler proof of the strong maximum principle.

Aim of this short note it show that this is actually the case: out of the several arguments available in the Euclidean space, the one  based on the estimates for the Laplacian of the squared distance carries over to such non-smooth context rather easily. 

Beside such Laplacian comparison, the other ingredient that we shall use is a result about a.e.\ unique projection on closed subsets of $\RCD$ spaces which to the best of our knowledge has not been observed before and, we believe, is of its own interest: see Lemma \ref{le:proj} and Remark \ref{re:proj}.
 %

We remark that the present result simplifies the proofs of those properties of $\RCD(K,N)$ spaces which depend on the strong maximum principle, like for instance the splitting theorem (\cite{Gigli13}, \cite{Gigli13over}).

\section{Result}
All metric measure spaces $\mms$ we will consider will be such that $(\Sp,\sfd)$ is complete and separable, $\mm$ is a Radon non-negative measure with $\supp(\mm)=\Sp$.

To keep the presentation short we assume the reader familiar with the definition of $\RCD^*$ spaces and with calculus on them. Here we only recall those definitions and facts that will be used in the course of the proofs. In particular, we shall take for granted the notion of $W^{1,2}(\Sp)$ space on the metric measure space $\mms$ and, for $f\in W^{1,2}(\Sp)$, of the minimal weak upper gradient $|\D f|$. Recall that the minimal weak upper gradient is a local object, i.e.:
\begin{equation}
\label{eq:locwug}
|\D f|=|\D g|\qquad\mm-a.e.\ \text{on }\{f=g\}\qquad\forall f,g\in W^{1,2}(\Sp).
\end{equation}
Then the notion of Sobolev space over an open set can be easily given: 
\begin{definition}[{\bf Sobolev space on an open subset of $\X$}]\label{def:WOmega}
Let $\mms$ be a metric measure space and let $\Omega \subset \X$ open. Then we define
\[\begin{split} \W_\loc(\Omega) := \{  f \in \Lint^2_{\loc}(\Omega) : \,\,&\text{for every} \,\,\, x \in \Omega \,\,\, \text{there exists} \,\, \, U \subset \Omega \,\,\text{neighbourhood of} \, \, \, x\\
 &\text{and there exists} \, \, \, f_U \in \W_\loc(\X) \,\,\, \text{such that} \,\,\, f\restr U = f_U \}.
\end{split}\]
For $f\in W^{1,2}_\loc(\Omega)$ the function $|\D f|\in L^2_\loc(\Omega)$ is defined as
\[
|\D f|:=|\D f_U|\quad\mm-\text{a.e.\ on }U,
\]
where $|\D f_U|$ is the minimal weak upper gradient of $f_U$ and the locality of this object ensures that $|\D f|$ is well defined.

Then  we set \[\W(\Omega) := \{ f \in \W_\loc(\Omega) : f, |{\D f}| \in \Lint^2(\Omega) \}.\]
\end{definition}
The definition of (sub/super)-harmonic functions can be given in terms of minimizers of the Dirichlet integral (see \cite{Bjorn-Bjorn11} for a throughout discussion on the topic):
\begin{definition}[{\bf Subharmonic/Superharmonic/Harmonic functions}]\label{def:harm}
Let $\mms$ be a metric measure space and $\Omega$ be an open subset in $\X$. We say that $f$ is subharmonic (resp. superharmonic) in $\Omega$ if $f \in \W(\Omega)$ and for any $g \in \W (\Omega)$, $g \le 0$ (resp. $g \ge 0$) with $\supp g \subset \Omega$, it holds
\begin{equation}\label{eq:subhvf}  \dfrac{1}{2} \int_{\Omega} |{\D f}|^2 \, \d \mea \le \dfrac{1}{2} \int_{\Omega} |{\D (f + g)}|^2 \, \d \mea .  \end{equation}
$f$ is harmonic in $\Omega$ if it is both subharmonic and superharmonic.
\end{definition}
On $\RCD(K,\infty)$ spaces, the weak maximum principle can be deduced directly from the definition of subharmonic function and the following property, proved in \cite{AmbrosioGigliSavare11-2}:
\begin{equation}
\label{eq:StL}
\begin{split}
&\text{Let  $\mms$ be $\RCD(K,\infty)$, $K\in\R$, and  $f\in W^{1,2}(\Sp)$ be such that $|\D f|\in L^\infty(\Sp)$.}\\
&\text{Then there exists $\tilde f=f$ $\mm$-a.e.\ such that $\Lip(\tilde f)\leq \|\D f|\|_{L^\infty}$.}
\end{split}
\end{equation}

We can now easily prove the following:
\begin{theorem}[{\bf Weak Maximum Principle}]\label{WMP}
Let $\mms$ be an $\RCD(K, \infty)$ space,  $\Omega \subset \Sp$ open and let $f \in \W(\Omega) \cap \C(\bar{\Omega})$ be subharmonic.
Then 
\begin{equation}
\label{eq:wmax}
\sup_{\Omega} f \le \sup_{\partial \Omega} f,
\end{equation}
to be intended as `$f$ is constant' in the case $ \Omega = \Sp$.
\end{theorem}
\begin{proof} We argue by contradiction. If \eqref{eq:wmax} does not hold, regardless of weather $\Omega$ coincides with $\Sp$ or not, we can find  $c<\sup_\Omega f$ such that the function
\[
\tilde f:=\min\{c,f\}
\]
agrees with $f$ on $\partial\Omega$. The locality of the differential grants that
\begin{equation}
\label{eq:tildef}
|\D\tilde f|=\nchi_{\{f<c\}}|\D f|
\end{equation}
and from the assumption that $f$ is subharmonic and the fact that $\tilde f\leq f$ we deduce that
\[
\int_\Omega|\D f|^2\,\d\mm\leq \int_\Omega|\D \tilde f|^2\,\d\mm\stackrel{\eqref{eq:tildef}}=\int_{\{f<c\}\cap\Omega}|\D f|^2\,\d\mm,
\]
which forces 
\begin{equation}
\label{eq:dz}
\text{$|\D f|=0$ $\mm$-a.e.\ on $\{f\geq c\}$.}
\end{equation}
Now consider the function $g:=\max\{c,\nchi_\Omega f\}$, notice that our assumptions grant that $g\in C(\Sp)$ and that the locality of the differential yields
\begin{equation}
\label{eq:dg}
|\D g|=\nchi_{\Omega\cap\{f>c\}}|\D f|\stackrel{\eqref{eq:dz}}=0.
\end{equation}
Hence property \eqref{eq:StL} gives that $g$ is constant, i.e.\ $f\leq c$ on $\Omega$. This contradicts our choice of $c$ and gives the conclusion.
\end{proof}
We remark that in the finite-dimensional case one could conclude from \eqref{eq:dg} by using the Poincar\'e inequality in place of property \eqref{eq:StL}.

\bigskip

To prove the strong maximum principle we need to recall few facts. The first is the concept of measure-valued Laplacian (see \cite{Gigli12}), for which we restrict the attention to  proper (=closed bounded sets are compact) and infinitesimally Hilbertian  (=$W^{1,2}(\Sp)$ is an Hilbert space, see \cite{Gigli12}) spaces:
\begin{definition}[{\bf Measure valued Laplacian}] Let $(\X,\sfd,\mm)$ be proper and infinitesimally Hilbertian,  $\Omega \subset \X$ open and $f \in \W(\Omega)$. We say that $f$ has a measure valued Laplacian in $\Omega$, and write $f\in D(\bd,\Omega)$, provided there exists a Radon measure, that we denote by $\bold{\Delta} f \restr{\Omega}$, such that for every $g \colon \Sp \rightarrow \real$ Lipschitz with  support compact and contained in $\Omega$ it holds
\begin{equation}
\displaystyle \int g \, \d \bold{\Delta} f \restr{\Omega} = - \int \langle \nabla f, \nabla g \rangle \, \d \mea.
\end{equation}
If $\Omega=\Sp$ we write $f\in D(\bd)$ and $\bd f$.
\end{definition}
Much like in the smooth case, it turns out that being subharmonic is equivalent to having non-negative Laplacian. This topic has been investigated in \cite{Gigli12} and \cite{Gigli-Mondino12}, here we report the proof of this fact because in  \cite{Gigli-Mondino12} it has been assumed the presence of a Poincar\'e inequality, while working on proper infinitesimally Hilbertian spaces allows to easily remove such assumption.
\begin{theorem}
Let $(\X,\sfd,\mm)$ be a proper infinitesimally Hilbertian space,  $\Omega \subset \X$ open and $f \in \W(\Omega)$.

Then $f$ is subharmonic (resp.\ superharmonic, resp.\ harmonic) if and only if $f\in D(\bd,\Omega)$ with $\bd f\restr\Omega\geq 0$ (resp.\ $\bd f\restr\Omega\leq 0$, resp.\ $\bd f\restr\Omega= 0$).
\end{theorem}
\begin{proof}\\
\noindent{\bf Only if} Let ${\rm LIP}_c(\Omega)\subset W^{1,2}(\Omega)$ be the space of Lipschitz functions with support compact and contained in $\Omega $. For $g\in {\rm LIP}_c(\Omega)$ non-positive and $\eps>0$ apply \eqref{eq:subhvf} with $\eps g$ in place of $g$ to deduce
\[
\int_\Omega|\D (f+\eps g)|^2-|\D f|^2\,\d\mm\geq 0
\]
and dividing by $\eps$ and letting $\eps\downarrow0$ we conclude
\[
\int_\Omega\la\nabla f,\nabla g\ra\,\d\mm\geq 0.
\]
In other words, the linear functional ${\rm LIP}_c(\Omega)\ni g\mapsto -\int_\Omega\la\nabla f,\nabla g\ra\,\d\mm$ is positive. It is then well known, see e.g.\ \cite[Theorem 7.11.3]{Bogachev07}, that the monotone extension of such functional to the space of continuous and compactly supported functions on $\Omega$ is uniquely represented by integration w.r.t.\ a non-negative measure, which is the claim. 

\noindent{\bf If} Recall from \cite{AmbrosioGigliSavare11} that on general metric measure spaces Lipschitz functions are dense in energy in $W^{1,2}$; since infinitesimally Hilbertianity implies uniform convexity of $W^{1,2}$, we see that in our case  they are dense in the $W^{1,2}-$norm. Then by truncation and cut-off argument we easily see that 
\begin{equation}
\label{eq:lipd}
\text{$\big\{g\in {\rm LIP}_c(\Omega):g\leq 0\big\}$\quad is $W^{1,2}-$dense in\quad $\big\{g\in W^{1,2}(\Omega):\ g\leq 0\ \supp(g)\subset \Omega\big\}$}.
\end{equation}
Now notice that the convexity of $g\mapsto\frac12\int_\Omega|\D g|^2\,\d\mm$ grants that for any $g\in W^{1,2}(\Omega)$ it holds
\[
|\D(f+g)|^2-|\D f|^2\geq \lim_{\eps\downarrow0}\frac{|\D(f+\eps g)|^2-|\D f|^2}\eps=2\la\nabla f,\nabla g\ra
\]
and thus from the assumption $\bd f\restr\Omega\geq 0$ we deduce that
\begin{equation}
\label{eq:magglip}
\int_\Omega |\D(f+g)|^2-|\D f|^2\,\d\mm\geq 0
\end{equation}
for every  $g\in {\rm LIP}_c(\Omega)$ non-positive. Taking \eqref{eq:lipd} into account we see that  \eqref{eq:magglip} also holds for any $g\in W^{1,2}(\Omega)$ non-negative with $\supp(g)\subset\Omega$, which is the thesis.
\end{proof}
For $x\in\Sp$ we write $\sfd_x$ for the function $y\mapsto\sfd(x,y)$.  We shall need the following two properties of the squared distance function valid on $\RCD(K,N)$ spaces, $N<\infty$:
\begin{align}
\label{diffdist}
\sfd_{x_0}^2\in W^{1,2}_\loc(\Sp)\quad\text{ and }\quad|{\D ( \dist_{x_0}^2 )}|^2 &= 2\dist_{x_0}^2 \quad \mae,\\
\label{lapdist}
\sfd_{x_0}^2\in D(\bd)\quad\text{ and }\ \quad\bd {\dist_{x_0}^2(x)}& \le {\ell}_{K, N}(\dist_{x_0})\mm,
\end{align}
where $\ell_{K,N}:[0,+\infty)\to[0,+\infty)$  is some continuous function depending only on $K,N$. Property \eqref{diffdist} can be seen as a consequence of  Cheeger's work \cite{Cheeger00}: recall that $\CD(K,N)$ spaces are doubling (\cite{Sturm06II}) and support a 1-2 weak Poincar\'e inequality (\cite{Rajala12}) and notice that, being geodesic, the local Lipschitz constant of $\sfd_x$ is identically 1. An alternative proof, more tailored to the $\RCD$ setting, passes through the fact that  $\sfd_{x_0}^2/2$ is $c$-concave and uses the regularity of $W_2$-geodesics, see for instance \cite{GT17} for the details of the proof.

The Laplacian comparison estimate \eqref{lapdist} is one of the main results in \cite{Gigli12}. Notice that in \cite{Gigli12} such inequality has been obtained in its sharp form, but for our purposes the above formulation is sufficient.

Beside these facts, we shall need the following geometric property of $\RCD$ spaces, which we believe is interesting on its own:
\begin{lemma}[a.e.\ unique projection]\label{le:proj}
Let $K\in\R$, $N\in[1,\infty)$, $\mms$  an $\RCD(K, N)$ space  and $C\subset\Sp$ a closed set. Then for $\mm$-a.e.\ $x\in\Sp$ there exists a unique $y\in C$ such that
\begin{equation}
\label{eq:proj}
\sfd(x,y)=\min_{z\in C}\sfd(x,z).
\end{equation}
\end{lemma}
\begin{proof} Existence follows trivially from the fact that $\Sp$ is proper. For  uniqueness define
\[
\varphi(x):=\inf_{z\in C}\frac{\sfd^2(x,z)}2=\psi^c(x)\qquad\text{ where }\qquad\psi(y):=\left\{\begin{array}{ll}
0,&\qquad\text{ if }y\in C,\\
-\infty,&\qquad\text{ if }y\in \Sp\setminus C.
\end{array}\right.
\]
Since $\varphi^c=\psi^{cc}\geq\psi$, if $x\in\Sp$ and $y\in C$ are such that \eqref{eq:proj} holds,  we have
\[
\varphi(x)+\varphi^c(y)\geq \varphi(x)+\psi(y)\stackrel{\eqref{eq:proj}}=\frac{\sfd^2(x,y)}{2},
\]
i.e.\ $y\in\partial^c\varphi(x)$. Conclude recalling that since  $\varphi$ is  $c$-concave and real valued,  Theorem 3.4 in \cite{GigliRajalaSturm13} grants that for $\mm$-a.e.\ $x$ there exists a unique $y\in\partial^c\varphi(x)$.
\end{proof}
\begin{remark}\label{re:proj}{\rm The simple proof of this lemma relies on quite delicate properties of $\RCD$ spaces, notice indeed that the conclusion  can fail on the more general $\CD(K,N)$ spaces. Consider for instance $\R^2$ equipped with the distance coming from the $L^\infty$ norm and the Lebesgue measure $\mathcal L^2$. This is a  $\CD(0,2)$ space, as shown in the last theorem in \cite{Villani09}. Then pick $C:=\{(z_1,z_2):z_1\geq 0\}$ and notice that for every $(x_1,x_2)\in \R^2$ with $x_1<0$ there are uncountably many minimizers in \eqref{eq:proj}.
}\fr\end{remark}

We can now prove the main result of this note:
\begin{theorem}[{\bf Strong Maximum Principle}]
Let $K\in\R$, $N\in[1,\infty)$ and $\mms$  an $\RCD(K, N)$ space. Let  $\Omega\subset \Sp$ be open and connected and  let $f \in \W(\Omega) \cap \C(\bar{\Omega})$ be subharmonic and such that for some  $\bar x \in \Omega$ it holds  $f(\bar x) = \max_{\bar{\Omega}} f$. Then $f$ is constant.
\end{theorem}

\begin{proof} Put $m := \sup_{\Omega} f$, $C := \{  x \in \bar\Omega : f(x) = m \}$ and define 
\[
\Omega':=\big\{x\in\Omega\setminus C\ :\ \sfd(x,C)<\sfd(x,\partial\Omega)\big\}.
\]
By assumption we know that $C\cap\Omega\neq \emptyset$ and that $\Omega$ is connected, thus since $C$ is closed, either $C\supset \Omega$, in which case we are done, or  $\partial C\cap\Omega\neq\emptyset$, in which case $\Omega'\neq\emptyset$. We now show that such second case cannot occur, thus concluding the proof.

Assume by contradiction that $\Omega'\neq\emptyset$, notice that $\Omega'$ is open and thus $\mm(\Omega')>0$. Hence by Lemma \ref{le:proj} we can find $x\in\Omega'$ and $y\in C$ such that \eqref{eq:proj} holds. Notice that the definition of $\Omega'$ grants that $y\in\Omega$, put $r:=\sfd(x,y)$ and define
\[
h(z):=e^{-A\sfd^2(z,x)}-e^{-Ar^2},
\]
where $A\gg1$ will be fixed later. By the chain rule for the measure-valued Laplacian (see \cite{Gigli12}) we have that $h\in D(\bd)$ with
\[
\bd h=A^2e^{-A\sfd^2_x}|\D \sfd_x^2|^2\,\mm-Ae^{-A\sfd^2_x}\bd\sfd_x^2\stackrel{\eqref{diffdist},\eqref{lapdist}}\geq 2e^{-A\sfd^2_x}\big(A^2 \sfd_x^2-A\ell_{K,N}(\sfd_x)\big)\mm.
\]
Hence we can, and will, choose  $A$ so big that $\bd h\restr{B_{r/2}(y)}\geq 0$. Now let  $r'<r/2$ be such that $B_{r'}(y)\subset \Omega$ and notice that for every $\eps>0$ the function $f_\eps:=f+\eps h$ is subharmonic in $B_{r'}(y)$ and thus according to Theorem \ref{WMP} we  have
\begin{equation}
\label{eq:contr}
f_\eps(y)\leq\sup_{\partial B_{r'}(y)}f_\eps,\qquad\forall \eps>0.
\end{equation}
Since $\{h<0\}=\X\setminus \bar B_r(x)$ and $h(y)=0$ we have
\begin{equation}
\label{eq:c1}
f_\eps(y)>f_\eps(z)\qquad\forall z\in {\partial B_{r'}(y)\setminus \bar B_r(x)},\ \forall \eps>0.
\end{equation} 
On the other hand, $\partial B_{r'}(y)\cap \bar B_r(x)$ is a compact set contained in $\Omega\setminus C$, hence by continuity and the definition of $C$ we have
\[
f(y)>\sup_{\partial B_{r'}(y)\cap \bar B_r(x)}f
\]
and thus for $\eps>0$ sufficiently small we also have
\[
f_\eps(y)>\sup_{\partial B_{r'}(y)\cap \bar B_r(x)}f_\eps.
\]
This inequality, \eqref{eq:c1} and the continuity of $f_\eps$ contradict \eqref{eq:contr}; the thesis follows.
\end{proof}

\def\cprime{$'$} \def\cprime{$'$}


\begin{thebibliography}{10}

\bibitem{AmbrosioGigliSavare11}
L.~Ambrosio, N.~Gigli, and G.~Savar{\'e}.
\newblock Calculus and heat flow in metric measure spaces and applications to
  spaces with {R}icci bounds from below.
\newblock {\em Invent. Math.}, 195(2):289--391, 2014.

\bibitem{AmbrosioGigliSavare11-2}
L.~Ambrosio, N.~Gigli, and G.~Savar{\'e}.
\newblock Metric measure spaces with {R}iemannian {R}icci curvature bounded
  from below.
\newblock {\em Duke Math. J.}, 163(7):1405--1490, 2014.

\bibitem{AmbrosioGigliSavare12}
L.~Ambrosio, N.~Gigli, and G.~Savar{\'e}.
\newblock Bakry-\'{E}mery curvature-dimension condition and {R}iemannian
  {R}icci curvature bounds.
\newblock {\em The Annals of Probability}, 43(1):339--404, 2015.

\bibitem{AmbrosioMondinoSavare13}
L.~Ambrosio, A.~Mondino, and G.~Savar{\'e}.
\newblock Nonlinear diffusion equations and curvature conditions in metric
  measure spaces.
\newblock Preprint, arXiv:1509.07273, 2015.

\bibitem{Bjorn-Bjorn11}
A.~Bj{\"o}rn and J.~Bj{\"o}rn.
\newblock {\em Nonlinear potential theory on metric spaces}, volume~17 of {\em
  EMS Tracts in Mathematics}.
\newblock European Mathematical Society (EMS), Z\"urich, 2011.

\bibitem{Bogachev07}
V.~Bogachev.
\newblock {\em Measure theory. {V}ol. {I}, {II}}.
\newblock Springer-Verlag, Berlin, 2007.

\bibitem{Cheeger00}
J.~Cheeger.
\newblock Differentiability of {L}ipschitz functions on metric measure spaces.
\newblock {\em Geom. Funct. Anal.}, 9(3):428--517, 1999.

\bibitem{Erbar-Kuwada-Sturm13}
M.~Erbar, K.~Kuwada, and K.-T. Sturm.
\newblock On the equivalence of the entropic curvature-dimension condition and
  {B}ochner's inequality on metric measure spaces.
\newblock {\em Inventiones mathematicae}, 201(3):1--79, 2014.

\bibitem{Gigli13}
N.~Gigli.
\newblock The splitting theorem in non-smooth context.
\newblock Preprint, arXiv:1302.5555, 2013.

\bibitem{Gigli13over}
N.~Gigli.
\newblock An overview of the proof of the splitting theorem in spaces with
  non-negative {R}icci curvature.
\newblock {\em Analysis and Geometry in Metric Spaces}, 2:169--213, 2014.

\bibitem{Gigli12}
N.~Gigli.
\newblock On the differential structure of metric measure spaces and
  applications.
\newblock {\em Mem. Amer. Math. Soc.}, 236(1113):vi+91, 2015.

\bibitem{Gigli-Mondino12}
N.~Gigli and A.~Mondino.
\newblock A {PDE} approach to nonlinear potential theory in metric measure
  spaces.
\newblock {\em J. Math. Pures Appl. (9)}, 100(4):505--534, 2013.

\bibitem{GigliRajalaSturm13}
N.~Gigli, T.~Rajala, and K.-T. Sturm.
\newblock Optimal {M}aps and {E}xponentiation on {F}inite-{D}imensional
  {S}paces with {R}icci {C}urvature {B}ounded from {B}elow.
\newblock {\em J. Geom. Anal.}, 26(4):2914--2929, 2016.

\bibitem{GT17}
N.~Gigli and L.~Tamanini.
\newblock Second order differentiation formula on compact {${\rm RCD}^*(K,N)$}
  spaces.
\newblock Preprint, arXiv:1701.03932.

\bibitem{Rajala12}
T.~Rajala.
\newblock Local {P}oincar\'e inequalities from stable curvature conditions on
  metric spaces.
\newblock {\em Calc. Var. Partial Differential Equations}, 44(3-4):477--494,
  2012.

\bibitem{Sturm06II}
K.-T. Sturm.
\newblock On the geometry of metric measure spaces. {II}.
\newblock {\em Acta Math.}, 196(1):133--177, 2006.

\bibitem{Villani09}
C.~Villani.
\newblock {\em Optimal transport. Old and new}, volume 338 of {\em Grundlehren
  der Mathematischen Wissenschaften}.
\newblock Springer-Verlag, Berlin, 2009.

\end{thebibliography}
\end{document}